\documentclass{amsart}

\usepackage{amsmath,amssymb,amsthm,amsfonts,mathrsfs}
\usepackage{amscd,stmaryrd}
\usepackage[usenames,dvipsnames]{color}
\usepackage{hyperref}
\usepackage{graphicx,subfigure}
\definecolor{nicegreen}{RGB}{0,180,0}
\hypersetup{colorlinks=true,citecolor=nicegreen,linkcolor=Red,urlcolor=blue}
\usepackage[divide={2.45cm,*,2.45cm}]{geometry}
\usepackage{enumerate}
\usepackage{enumitem}
\usepackage{color}
\usepackage[all]{xy}
\usepackage{tikz-cd}
\usepackage{wasysym}
\usepackage{mathtools}

\allowdisplaybreaks

\newtheorem{thm}{Theorem}[section]
\newtheorem*{thm*}{Theorem}
\newtheorem{cor}[thm]{Corollary}
\newtheorem{lemma}[thm]{Lemma}

\newtheorem*{propn*}{Proposition}

\theoremstyle{definition}

\newtheorem{rmk}[thm]{Remark}

\newcommand{\soc}{\textnormal{soc}}

\newcommand{\End}{\textnormal{End}}
\newcommand{\Ind}{\textnormal{Ind}}
\newcommand{\cind}{\textnormal{c-ind}}

\newcommand{\der}{{\textnormal{der}}}

\newcommand*{\longhookrightarrow}{\ensuremath{\lhook\joinrel\relbar\joinrel\rightarrow}}
\newcommand*{\longtwoheadrightarrow}{\ensuremath{\relbar\joinrel\twoheadrightarrow}}

\newcommand{\cH}{{\mathcal{H}}}

\newcommand{\cO}{{\mathcal{O}}}

\newcommand{\bB}{{\mathbf{B}}}

\newcommand{\bG}{{\mathbf{G}}}

\newcommand{\bM}{{\mathbf{M}}}
\newcommand{\bN}{{\mathbf{N}}}

\newcommand{\bP}{{\mathbf{P}}}

\newcommand{\bS}{{\mathbf{S}}}
\newcommand{\bT}{{\mathbf{T}}}

\newcommand{\bZ}{{\mathbf{Z}}}

\newcommand{\bbF}{{\mathbb{F}}}

\newcommand{\bbQ}{{\mathbb{Q}}}

\newcommand{\bbZ}{{\mathbb{Z}}}

\newcommand{\nA}{\textnormal{A}}
\newcommand{\nB}{\textnormal{B}}
\newcommand{\nC}{\textnormal{C}}
\newcommand{\nD}{\textnormal{D}}
\newcommand{\nE}{\textnormal{E}}

\newcommand{\nG}{\textnormal{G}}
\newcommand{\nH}{\textnormal{H}}

\newcommand{\nL}{\textnormal{L}}

\newcommand{\nS}{\textnormal{S}}

\begin{document}
\nocite{}

\title{Non-equivalence of pro-$p$-Iwahori invariants}
\date{}
\author{Karol Kozio\l}
\thanks{The author was supported by a PSC-CUNY Trad B award.  Google Gemini was used for proofreading and checking mathematical arguments. The mathematical content is due to the author.}
\address{Baruch College, City University of New York, 55 Lexington Ave., New York, NY 10010} 
\email{karol.koziol@baruch.cuny.edu}


\begin{abstract}
  Suppose $p > 7$ is a prime number, and $F$ is a nonarchimedean local field of residual characteristic $p$ satisfying $F \neq \mathbb{Q}_p$.  Generalizing results of Ghate--Le--Sheth, Le--Singh, and Huang, we show that any split, connected, reductive group $G$ over $F$ which is not a torus admits a smooth, irreducible, non-admissible mod $p$ representation.  We use this to show that the functor of pro-$p$-Iwahori invariants does not induce an equivalence between the category of smooth mod $p$ $G$-representations generated by their pro-$p$-Iwahori invariant vectors and modules over the pro-$p$-Iwahori--Hecke algebra, contrary to what happens for $\textnormal{GL}_2(\mathbb{Q}_p)$.  
\end{abstract}

\maketitle

\section{Introduction}
Suppose $F$ is a nonarchimedean local field of residual characteristic $p$, $G$ the group of $F$-points of a split, connected, reductive group over $F$, and $k$ a coefficient field of characteristic $p$.  In order to analyze the category $\mathfrak{Rep}_k(G)$ of smooth $G$-representations on $k$-vector spaces, it is often advantageous to restrict the action to compact open subgroups of $G$.  In particular, taking $I_1$ to be a choice of pro-$p$-Iwahori subgroup, the full subcategory $\mathfrak{Rep}^{I_1}_k(G)$ of objects generated by their $I_1$-invariant vectors has proven to be particularly useful, owing to the fact that any nonzero smooth $G$-representation $\pi$ over $k$ satisfies $\pi^{I_1} \neq 0$.  Moreover, we obtain a pair of adjoint functors
\begin{eqnarray}
\mathfrak{Rep}_k^{I_1}(G) & \begin{tikzcd}
\arrow[r, shift left] & \arrow[l, shift left]
\end{tikzcd}
 & \mathfrak{Mod-}\cH_k(G,I_1) \label{equiv}\\
\pi & \longmapsto & \pi^{I_1}\notag \\ 
M\otimes_{\cH_k(G,I_1)}\cind_{I_1}^G(k)  & \longmapsfrom & M \notag
\end{eqnarray}
where $\cH_k(G,I_1) := k[I_1\backslash G/I_1]$ denotes the pro-$p$-Iwahori--Hecke algebra.  This viewpoint has found important applications to the mod $p$ Local Langlands Program (for example, see \cite[Thm. 0.7, Rmk. 0.8]{colmez:res-phi-gamma-mods} and \cite{grosseklonne:phi-gamma}).

When $G = \nG\nL_2(\bbQ_p)$ or $\nS\nL_2(\bbQ_p)$, the adjoint functors \eqref{equiv} induce an equivalence of categories (\cite{ollivier:foncteur}, \cite{koziol:glnsln}, \cite{ollivierschneider:torsion}).  Furthermore, it is known that we also obtain an equivalence for a general $G$ if we restrict to ``principal series'' subcategories on both sides (see \cite{ollivierschneider:torsion}, \cite{abe:comparison} for precise statements).  Despite these positive results, \cite[\S 1, Thm.]{ollivier:foncteur} shows that we cannot expect an equivalence in general, at least if $G = \nG\nL_2(F)$.

The purpose of this note is to prove that for ``most'' split $p$-adic reductive groups, the functors \eqref{equiv} fail to give an equivalence.  More precisely, we suppose $p > 3$ and that the residue field of $F$ is a proper extension of $\bbF_p$, or that $p > 7$ and that the residue field of $F$ is $\bbF_p$ with $F \neq \bbQ_p$.  We build on the results of \cite{GLS}, \cite{le-singh}, and \cite{huang} to show that non-torus groups $G$ as above admit smooth, irreducible, non-admissible representations defined over $k$ (Corollary \ref{maincor}).  The existence of such a representation implies that \eqref{equiv} cannot induce an equivalence (Theorem \ref{thm:hecke-alg-cor}).  We use this to deduce a converse to the theorem of Ollivier mentioned above.

\section{Notation and recollections}
\label{sec:recoll}

 We fix a nonarchimedean local field $F$ of residual characteristic $p$ and residue field $k_F$ of degree $f$ over $\bbF_p$.  We assume throughout this note that we are in one of the following two cases:
 \begin{enumerate}[label=(\Alph*), ref=\Alph*]
  \item We have $f > 1$ and $p > 3$. \label{caseA}
  \item We have $f = 1$, $p > 7$, and $F \neq \bbQ_p$.  \label{caseB}
 \end{enumerate}
  We also let $\cO_F$ denote the ring of integers of $F$, and $\varpi \in \cO_F$ a choice of uniformizer.  Note that we allow the characteristic of $F$ to be either $0$ or $p$.

We first recall several results about non-admissible representations of $\nG\nL_2(F)$ from \cite{GLS} and \cite{le-singh}.  Let $\boldsymbol{r} = (r_0, r_1, \ldots, r_{f - 1}) \in \bbZ_{\geq 0}^f$ denote an $f$-tuple satisfying $1 \leq r_j \leq p - 3$ for all $j$ in case \eqref{caseA}, and $5 \leq r_0 \leq p - 2$ in case \eqref{caseB}.  We also fix $\lambda = (\lambda_i)_{i \in \bbZ} \in \prod_{i \in \bbZ} k_F^\times$ satisfying the conditions of \cite[Prop. 3.1]{GLS:erratum} in case \eqref{caseA}, and the conditions of \cite[Prop. 3.3]{le-singh} in case \eqref{caseB} (for $\xi$ trivial).  We then let $\rho$ denote the representation of $\nG\nL_2(F)$ over $\overline{k_F}$ constructed in \cite[\S 4, pf. of Thm. 1.1]{GLS} (see also \cite[Pf. of Thm. 1.1]{GLS:erratum}) in case \eqref{caseA}, and in \cite[Thm. 4.1]{le-singh} in case \eqref{caseB}.  The representation $\rho$ satisfies the following properties:
\begin{itemize}[label=$\diamond$]
  \item $\rho$ is smooth, irreducible, and non-admissible.
  \item $\rho$ possesses a central character, given by 
  $$\begin{pmatrix}a\varpi^v & 0 \\ 0 & a\varpi^v\end{pmatrix} \longmapsto \overline{a}^{\sum_{j = 0}^{f - 1}r_jp^j}$$
  ($a \in \cO_F^\times, v \in \bbZ$).
  \item $\rho$ has a model $\rho_0$ over $k_F$, which is smooth, absolutely irreducible, and non-admissible.  In particular, $\End_{k_F[\nG\nL_2(F)]}(\rho_0) = k_F$.
\end{itemize}

  Suppose now that $\bG$ is a split connected reductive group over $F$, and fix a maximal $F$-split torus $\bT$ and Borel subgroup $\bB$ containing $\bT$.  We assume throughout that $\bG \neq \bT$.  For a simple root $\alpha \in \Delta \subset \Phi(\bG,\bT)$ (relative to $\bB$), we let $\bM_\alpha$ denote the closed subgroup of $\bG$ generated by $\bT$ and the root subgroups corresponding to $\pm \alpha$, let $\bP_\alpha = \bM_\alpha \bN_\alpha$ denote the closed subgroup generated by $\bM_\alpha$ and $\bB$, and $\bP_\alpha^- = \bM_\alpha \bN_\alpha^-$ its opposite parabolic (where $\bN_\alpha$ and $\bN_\alpha^-$ denote the unipotent radicals of $\bP_\alpha$ and $\bP_\alpha^-$, respectively).  We denote by italic letters the groups of $F$-points of algebraic groups (so that $G = \bG(F), T = \bT(F)$, etc.).  Moreover, since the group $\bG$ is split over $F$, it and all the subgroups above have smooth integral models over $\cO_F$, which by abuse of notation we denote by the same letters.

\section{Non-admissible representations}

  Our goal is to construct a smooth, irreducible, non-admissible representation of $G$ defined over $\overline{k_F}$.  We begin with the following.

  \begin{lemma}
    \label{lem1}
    Suppose $\bG$ is almost simple and simply connected, and satisfies $\bG \neq \bT$.  Then $G$ possesses a smooth, irreducible, non-admissible representation $\pi$ defined over $\overline{k_F}$.  Moreover, the representation $\pi$ can be chosen so that the (finite) center of $G$ acts trivially.
  \end{lemma}

  \begin{proof}
    We proceed in several steps.

    \vspace{5pt}

    \textit{Step 1.}  We prove the claim when $\bG$ is of rank 1 (that is, type $\nA_1$), so that  $G \cong \nS\nL_2(F)$.  Let $Z_{\nG\nL_2} \cong F^\times$ denote the center of $\nG\nL_2(F)$.  Since $p > 3$, $Z_{\nG\nL_2}\nS\nL_2(F)$ is a normal subgroup of $\nG\nL_2(F)$ of index 4.  Let $\rho$ denote an irreducible non-admissible representation of $\nG\nL_2(F)$ as in Section \ref{sec:recoll}.  Then the restriction $\rho|_{Z_{\nG\nL_2}\nS\nL_2(F)}$ splits as the direct sum of at most 4 irreducible $Z_{\nG\nL_2}\nS\nL_2(F)$-representations (see, for example, \cite[Lem. 2.4]{labesselanglands} or \cite{MO-clifford}, the proofs of which work over general coefficient fields).  Since $\rho$ possesses a central character, the irreducible summands of $\rho|_{Z_{\nG\nL_2}\nS\nL_2(F)}$ will remain irreducible upon further restriction to $\nS\nL_2(F)$.  Next, if we let $I_1$ denote a pro-$p$-Iwahori subgroup of $\nG\nL_2(F)$, then the assumption $p > 3$ implies $I_1 = (Z_{\nG\nL_2} \cap I_1)(\nS\nL_2(F)\cap I_1)$, which gives $\rho^{I_1} = \rho^{\nS\nL_2(F) \cap I_1}$.  Thus, at least one of the irreducible summands of $\rho|_{\nS\nL_2(F)}$ must be non-admissible, and we denote this summand by $\pi$.  Finally, if $\boldsymbol{r}$ is chosen such that all $r_j$ are even, then the center of $\nS\nL_2(F)$ will act trivially on $\rho$, and hence also on $\pi$.  

    \vspace{5pt}

    \textit{Step 2.}  We suppose in the remainder of the proof that $\bG$ has rank $n \geq 2$, and show that we can select $\alpha \in \Delta$ for which $\bM_\alpha \cong \nG\nL_2 \times \bT'$ for some split torus $\bT'$.  Let $\alpha$ denote a short simple root which is extremal in the Dynkin diagram of $\bG$, let $\beta$ denote its unique neighbor, and enumerate the remaining roots as $\Delta \smallsetminus \{\alpha, \beta\} = \{\gamma_3, \ldots, \gamma_n\}$.  Since $\alpha$ is short and extremal, we have $\langle \alpha,\beta^\vee\rangle = -1$ and $\langle \alpha,  \gamma_i^\vee\rangle = 0$ for $i = 3, \ldots, n$.  Furthermore, since $\bG$ is simply connected, we have $X_*(\bT) = \bigoplus_{\gamma \in \Delta} \bbZ \gamma^\vee$ and $X^*(\bT) = \bigoplus_{\gamma \in \Delta} \bbZ \varpi_\gamma$, where $\varpi_\gamma$ denotes the fundamental dominant weight associated to $\gamma$ (i.e., $\langle \varpi_\gamma, \eta^\vee\rangle = \delta_{\gamma, \eta}$ for all $\gamma, \eta \in \Delta$).  Note that by our choice of $\alpha$, we have $\alpha = 2\varpi_\alpha - \varpi_\beta$.

    We claim that $\bM_\alpha \cong \nG\nL_2 \times \bT'$, where $\bT'$ denotes the subtorus of $\bT$ generated by the images of $\gamma_3^\vee, \ldots, \gamma_n^\vee$.  Indeed, it suffices to construct an isomorphism $\varphi:X^*(\bT) \stackrel{\sim}{\longrightarrow} \bbZ^{\oplus n}$ taking $\alpha$ to $(1,-1,0, \ldots, 0)$, and whose transpose $\varphi^\top:\bbZ^{\oplus n} \stackrel{\sim}{\longrightarrow} X_*(\bT)$ sends $(1,-1,0,\ldots,0)$ to $\alpha^\vee$.  We define $\varphi$ by
    $$\varphi(\varpi_\alpha) = e_1, \qquad \varphi(\varpi_\beta) = e_1 + e_2, \qquad \varphi(\varpi_{\gamma_i}) = e_i~\textnormal{for $i = 3, \ldots, n$}.$$
    One then verifies that 
    $$\varphi^\top(e_1^*) = \alpha^\vee + \beta^\vee, \qquad \varphi^\top(e_2^*) = \beta^\vee, \qquad \varphi^\top(e_i^*) = \gamma_i^\vee ~\textnormal{for $i = 3, \ldots, n$},$$
  and 
  $$\varphi(\alpha) = 2\varphi(\varpi_\alpha) - \varphi(\varpi_\beta) = e_1 - e_2,\qquad \varphi^\top(e_1^* - e_2^*) = \varphi^\top(e_1^*) - \varphi^\top(e_2^*) = \alpha^\vee,$$
  as desired.

    \vspace{5pt}

    \textit{Step 3.}  We construct our candidate representation $\pi$.  Let us choose $\xi \in X^*(\bT')$ satisfying $0 < \langle \xi, e_i^*\rangle < p^f - 1$ for $i = 3, \ldots, n$ (e.g., $\xi = \sum_{i = 3}^n e_i$), and let $\chi':T' \longrightarrow k_F^\times \longhookrightarrow \overline{k_F}^\times$ denote any smooth character for which $\chi'|_{\bT'(\cO_F)}$ is equal to the inflation to $\bT'(\cO_F)$ of $\xi|_{\bT'(k_F)}$.  We also choose a smooth, non-admissible representation $\rho$ of $\nG\nL_2(F)$ as in Section \ref{sec:recoll}, and select an integer $0 \leq m < p^f - 1$ such that $m$ is distinct modulo $p^f - 1$ from all determinant powers appearing in $\soc_{\nG\nL_2(\cO_F)}(\rho)$ (such a choice is possible since the number of distinct weights in $\soc_{\nG\nL_2(\cO_F)}(\rho)$ is at most $4f - 1 < p^f - 1$ in case \eqref{caseA}, and at most $6 < p - 1$ in case \eqref{caseB}; see \cite[pf. of Lem. 4.5]{GLS} and \cite[\S 3]{le-singh}).  We set $\rho' := \rho \otimes (\omega^{-m}\circ \det_{\nG\nL_2})$, where $\omega:F^\times \longrightarrow k_F^\times \longhookrightarrow \overline{k_F}^\times$ is the character $\omega(a\varpi^v) = \overline{a}$ for $a \in \cO_F^\times, v \in \bbZ$.  Via the isomorphism $\varphi:  \nG\nL_2 \times \bT' \stackrel{\sim}{\longrightarrow} \bM_\alpha$, we view $\rho' \boxtimes \chi'$ as a smooth representation of $M_\alpha$, which we then inflate to $P_\alpha^-$, and denote by $\pi := \Ind_{P_\alpha^-}^G(\rho'\boxtimes \chi')$ the parabolically induced representation.  As in \cite[pf. of Thm. 1.1]{GLS}, the representation $\pi$ is smooth and non-admissible.

    \vspace{5pt}

    \textit{Step 4.} We prove that $\pi$ is irreducible.  To this end, we claim that every $\bG(\cO_F)$-weight of $\pi$ is $M_\alpha$-regular in the sense of \cite[Def. 2.4]{herzig:gln}.  By \cite[Prop. 2.2]{herzig:gln}, any $\bG(\cO_F)$-weight of $\pi$ is of the form $F(\nu)$ for $\nu \in X^*(\bT)$ satisfying $0 \leq \langle \nu, \gamma^\vee\rangle \leq p^f - 1$ for all $\gamma \in \Delta$.  As in the proof of \cite[Lem. 4.1]{GLS} (and using \cite[Lem. 2.3]{herzig:gln}), the inclusion $F(\nu) \longhookrightarrow \pi|_{\bG(\cO_F)}$ gives rise to an inclusion $F^{M_\alpha}(\nu) \cong F(\nu)^{\bN_\alpha(\cO_F)} \longhookrightarrow (\rho'\boxtimes \chi')|_{\bM_\alpha(\cO_F)}$.  Applying the isomorphism $\varphi$ gives 
    $$F^{\nG\nL_2(F) \times T'}(\varphi(\nu)) \cong F^{\nG\nL_2(F)}(\textnormal{pr}_{\nG\nL_2}(\nu))\boxtimes (\textnormal{pr}_{\bT'}(\nu)|_{\bT'(k_F)}) \longhookrightarrow (\rho'|_{\nG\nL_2(\cO_F)})\boxtimes (\chi'|_{\bT'(\cO_F)}),$$ 
    where $\textnormal{pr}_{\nG\nL_2}$ (resp., $\textnormal{pr}_{\bT'}$) denotes the composition of $\varphi:X^*(\bT) \stackrel{\sim}{\longrightarrow} \bbZ^{\oplus n}$ with the projection onto the first $2$ components (resp., last $n - 2$ components).  Hence, we see that $F^{\nG\nL_2(F)}(\textnormal{pr}_{\nG\nL_2}(\nu))$ is a $\nG\nL_2(\cO_F)$-weight of $\rho'$, and $\textnormal{pr}_{\bT'}(\nu)|_{\bT'(k_F)} \cong \chi'|_{\bT'(\cO_F)} \cong \xi|_{\bT'(k_F)}$.  Note that for $i = 3, \ldots, n$, we have 
    $$\langle \nu, \gamma_i^\vee\rangle = \langle \nu, \varphi^\top(e_i^*) \rangle = \langle \varphi(\nu), e_i^* \rangle = \langle \textnormal{pr}_{\bT'}(\nu), e_i^* \rangle = \langle \xi, e_i^* \rangle.$$
  Thus, we obtain $0 < \langle \nu, \gamma_i^\vee\rangle < p^f - 1$ for $i = 3, \ldots, n$.  Similarly, we have
      $$\langle \nu, \beta^\vee\rangle = \langle \nu, \varphi^\top(e_2^*) \rangle = \langle \varphi(\nu), e_2^* \rangle = \langle \textnormal{pr}_{\nG\nL_2}(\nu), e_2^* \rangle.$$
      The latter quantity is equal to the determinant power appearing in the weight $F^{\nG\nL_2(F)}(\textnormal{pr}_{\nG\nL_2}(\nu))$, which by our choice of $m$ is never equal to 0.  This implies $0 < \langle \nu,\beta^\vee\rangle \leq p^f - 1$, and we obtain that $F(\nu)$ is $M_\alpha$-regular by the discussion preceding \cite[Lem. 2.5]{herzig:gln}.  Using the arguments of \cite[Lem. 4.1]{GLS} and \cite[Thm. 4.1(2)]{le-singh}, we conclude that $\pi$ is irreducible.

    \vspace{5pt}

    \textit{Step 5.} Finally, we check that $\pi$ may be chosen to have trivial central character.  Since the center of $G$ is finite and contained in $T$, any central element is of the form $\lambda(\zeta)$, where $\lambda \in X_*(\bT)$ and $\zeta \in F$ is a root of unity.  In the cases below, we use the notation $\textnormal{pr}_{\nG\nL_2}$ (resp., $\textnormal{pr}_{\bT'}$) to denote the composition of $\varphi^{\top, -1}:X_*(\bT) \stackrel{\sim}{\longrightarrow} \bbZ^{\oplus n}$ with the projection onto the first $2$ components (resp., the last $n - 2$ components).   
      \begin{itemize}[label=$\diamond$]
        \item Suppose first that $\bG$ is not of type $\nA_n$.  Then the center $Z$ of $G$ has order at most $4$ (see \cite[Ch. VI, Planches II -- IX]{bourbaki:lie}), and any element of $Z$ is of the form $\lambda(\zeta)$ where the order of $\zeta$ divides $|Z|$.  This element acts on $\pi$ via the scalar 
        $$\omega_{\rho'}\left(\textnormal{pr}_{\nG\nL_2}(\lambda)(\zeta)\right)\chi'\left(\textnormal{pr}_{\bT'}(\lambda)(\zeta)\right),$$
        where $\omega_{\rho'}$ denotes the central character of $\rho'$.  The second factor equals $\overline{\zeta}^{\langle \xi, \textnormal{pr}_{\bT'}(\lambda)\rangle}$; choosing $\xi$ such that $\langle \xi, e_i^*\rangle = |Z|$ for $i = 3, \ldots, n$ gives $\chi'\left(\textnormal{pr}_{\bT'}(\lambda)(\zeta)\right) = 1$.  The first factor is equal to a power of $\overline{\zeta}^{~-2m + \sum_{j = 0}^{f - 1}r_jp^j}$; proceeding in a case-by-case manner shows that we may always choose $\boldsymbol{r} = (r_0, \ldots, r_{f - 1})$ and $m$ to make this quantity equal to 1.  (For example, if $\bG$ is of type $\nB_n, \nC_n, \nD_n$ with $n$ even, or $\nE_7$, every element of $Z$ has order 2, and it suffices to choose all $r_j$ to be even.)
        \item Suppose $\bG$ is of type $\nA_n$, so that $G \cong \nS\nL_{n + 1}(F)$.  The center $Z$ of $G$ is then generated by $\lambda(\zeta)$, where $\zeta$ is a generator of $\mu_{n + 1}(F)$ and $\lambda = \sum_{i = 1}^n i\alpha_i^\vee$ (using the numbering in \cite[Ch. VI, Planche I]{bourbaki:lie}, where we take the $\alpha$ above to be equal to $\alpha_1$).  This element acts on $\pi$ via the scalar
        $$\omega_{\rho'}\left(\textnormal{pr}_{\nG\nL_2}(\lambda)(\zeta)\right)\chi'\left(\textnormal{pr}_{\bT'}(\lambda)(\zeta)\right) = \overline{\zeta}^{~ -2m + \sum_{j = 0}^{f - 1}r_jp^j + \sum_{i = 3}^n i\langle\xi, e_i^* \rangle}.$$
        We leave it as an exercise to the reader to verify that we may always choose $\boldsymbol{r} = (r_0, \ldots, r_{f - 1})$, $m$, and $\langle \xi, e_i^*\rangle$ ($i = 3, \ldots, n$) such that $-2m + \sum_{j = 0}^{f - 1}r_jp^j + \sum_{i = 3}^n i\langle\xi, e_i^* \rangle \equiv 0~(\textnormal{mod}~ n + 1)$, which finishes the claim.
        \end{itemize}
\end{proof}

\begin{lemma}
  \label{lem2}
  Suppose $\bG$ has simply connected derived subgroup, and satisfies $\bG \neq \bT$.  Then $G$ possesses a smooth, irreducible, non-admissible representation $\pi$ defined over $\overline{k_F}$.  Moreover, the representation $\pi$ can be chosen so that the $F$-points of the connected center of $\bG$ act trivially.
\end{lemma}

\begin{proof}
  Let $\bG^\der$ denote the derived subgroup of $\bG$ and let $\bZ^\circ$ denote the connected center of $\bG$ (by which we mean the maximal central torus of $\bG$).  By the simply connected assumption on $\bG^\der$, we have a factorization
  \begin{equation}
    \label{Gderfactor}
    \bG^\der \cong \bG_1 \times \bG_2 \times \cdots \times \bG_n,
  \end{equation}
  where each $\bG_i$ is almost simple and simply connected.  We again proceed in several steps.

  \vspace{5pt}

  \textit{Step 1.}  Suppose first that every factor of $\bG^\der$ is of type $\nA_1$.  In this case, we consider the short exact sequence
  \begin{equation}
    \label{center-derived}
    1 \longrightarrow \bZ^\circ \cap \bG^\der \longrightarrow \bZ^\circ \times \bG^\der \stackrel{m}{\longrightarrow} \bG \longrightarrow 1,
  \end{equation}
  where $\bZ^\circ \cap \bG^\der$ is a finite, central subgroup embedded anti-diagonally in $\bZ^\circ \times \bG^\der$, and $m$ denotes the multiplication map.

  Define a representation $\pi'$ of $Z^\circ \times G^\der$ by
  $$\pi' := 1_{Z^\circ} \boxtimes \pi_{G_1} \boxtimes 1_{G_2} \boxtimes \cdots \boxtimes 1_{G_n},$$
  where $1_H$ denotes the trivial representation of $H$ over $\overline{k_F}$, and where $\pi_{G_1}$ denotes a smooth, irreducible, non-admissible representation of $G_1 \cong \nS\nL_2(F)$ with trivial action of the finite group $Z_{G_1}$ constructed in Lemma \ref{lem1}.  The representation $\pi'$ is smooth, irreducible, non-admissible, and has a trivial action of $Z^\circ$ by definition.  Note also that $\bZ^\circ \cap \bG^\der$ is contained in the center $\bZ_{\bG^\der} \cong \bZ_{\bG_1} \times \cdots \times \bZ_{\bG_n}$ (whose order is a power of 2).  Therefore, by our choice of $\pi_{G_1}$ we see that the group $Z^\circ \cap G^\der$ acts trivially on $\pi'$, so that $\pi'$ descends to a representation of $Z^\circ G^\der$.

  We note that our assumption $p > 3$ implies $\bZ^\circ \cap \bG^\der$ is a smooth group scheme over $F$.  Taking $F$-rational points of \eqref{center-derived} gives an exact sequence 
  $$1 \longrightarrow Z^\circ \cap G^\der \longrightarrow Z^\circ \times G^\der \stackrel{m}{\longrightarrow} G \longrightarrow \nH^1(F, \bZ^\circ \cap \bG^\der).$$
  By \cite[Thm. 7.1.8(iii)]{nsw:coh}, the cohomology group $\nH^1(F, \bZ^\circ \cap \bG^\der)$ is finite, and therefore $Z^\circ G^\der$ is of finite index in $G$.  Thus, the $G$-representation $\Ind_{Z^\circ G^\der}^G(\pi')$ is a smooth, non-admissible $G$-representation of finite length.  Further, at least one Jordan--H\"older factor of $\Ind_{Z^\circ G^\der}^G(\pi')$ must be non-admissible.  Defining $\pi$ to be this subquotient gives a smooth, irreducible, non-admissible $G$-representation with trivial action of $Z^\circ$.

  \vspace{5pt}

  \textit{Step 2.}  We suppose in the remainder of the proof that $\bG^\der$ contains a factor which is not of type $\nA_1$.  We shall enlarge the closed subgroup $Z^\circ G^\der$ of $G$ into an open subgroup of finite index.

  We enumerate $\bG^\der$ as in \eqref{Gderfactor} so that $\bG_1$ is not of type $\nA_1$.  Let $\alpha$ denote a short simple root which is extremal in the Dynkin diagram of $\bG_1$ (as in the proof of Lemma \ref{lem1}), and let $\bT_\alpha := \ker(\alpha)^\circ$ denote the connected kernel of $\alpha:\bT \longrightarrow \bG_m$, a subtorus of $\bT$.  Consider the short exact sequence 
  \begin{equation}
    \label{eqn:coradical}
    1 \longrightarrow \bG^\der \longrightarrow\bG \longrightarrow \bS \longrightarrow 1,
  \end{equation}
  where $\bS := \bG/\bG^\der \cong \bT/\bT^\der$ is a split torus (here $\bT^\der := \bT \cap \bG^\der$).  The composition $\bT_\alpha \longhookrightarrow \bT \longtwoheadrightarrow \bS$ induces morphisms of cocharacter lattices $X_*(\bT_\alpha) \longhookrightarrow X_*(\bT) \longtwoheadrightarrow X_*(\bS)$, and we claim that the composition $X_*(\bT_\alpha) \longrightarrow X_*(\bS)$ is surjective.  Since $\bT^\der$ is a maximal torus of $\bG^\der$ (and therefore connected), the kernel of the projection $X_*(\bT) \longtwoheadrightarrow X_*(\bS)$ is exactly $X_*(\bT^\der)$.  Further, we note that 
  $$X_*(\bT^\der) = \bigoplus_{\gamma \in \Delta} \bbZ \gamma^\vee \qquad\textnormal{and}\qquad X_*(\bT_\alpha) = \{\lambda \in X_*(\bT): \langle \alpha, \lambda\rangle = 0\}.$$
  Consequently, letting $\beta \in \Delta$ denote the unique neighbor of $\alpha$ and $\lambda \in X_*(\bT)$, we obtain $\lambda = (\lambda + \langle \alpha, \lambda\rangle\beta^\vee) - \langle \alpha, \lambda\rangle\beta^\vee \in X_*(\bT_\alpha) + X_*(\bT^\der)$.  Therefore, the image of $\lambda$ in $X_*(\bS)$ is the same as the image of $\lambda + \langle \alpha, \lambda\rangle \beta^\vee$, from which we conclude that $X_*(\bT_\alpha) \longrightarrow X_*(\bS)$ is surjective.

  Let $T_{\alpha,1} := \ker(\bT_\alpha(\cO_F) \longtwoheadrightarrow \bT_\alpha(k_F))$.  We claim that $T_{\alpha,1}Z^\circ G^\der$ is of finite index in $G$.  By \cite[Thm. 25.61]{milne:alggrp}, taking $F$-rational points of \eqref{eqn:coradical} gives an exact sequence of topological groups
  $$1 \longrightarrow G^\der \longrightarrow G \longrightarrow S \longrightarrow 1.$$
  It therefore suffices to show the image of $T_{\alpha,1}Z^\circ$ in $S$ is of finite index.  We note first that the surjection $X_*(\bT_\alpha) \longtwoheadrightarrow X_*(\bS)$ induces a surjection 
  $$X_*(\bT_\alpha) \otimes_{\bbZ}(1 + \varpi\cO_F) \longtwoheadrightarrow X_*(\bS)\otimes_{\bbZ}(1 + \varpi\cO_F).$$
  Using the isomorphism $X_*(\bT_\alpha) \otimes_{\bbZ}F^\times \stackrel{\sim}{\longrightarrow} T_\alpha$ (and likewise for $\bS$), the above surjection implies that $T_{\alpha,1}$ maps surjectively onto $S_1$, the pro-$p$ part of $\bS(\cO_F)$.  Thus, we are reduced to showing $Z^\circ$ has image of finite index in $S/S_1$, or equivalently in $S/\bS(\cO_F) \cong X_*(\bS) \otimes_{\bbZ} \varpi^{\bbZ}$.  To see this, note that the composition $\bZ^\circ \longhookrightarrow \bG \longtwoheadrightarrow \bS$ is an isogeny, and therefore the induced map $X_*(\bZ^\circ) \longrightarrow X_*(\bS)$ is injective with image of finite index.  Thus, $Z^{\circ}/\bZ^\circ(\cO_F)$ has image of finite index in $S/\bS(\cO_F)$, and the claim follows.

  \vspace{5pt}

  \textit{Step 3.} We construct a representation $\pi'$ of $G^\der$ having the desired properties.  Define $\pi'$ by
  $$\pi' := \pi_{G_1} \boxtimes 1_{G_2} \boxtimes \cdots \boxtimes 1_{G_n},$$
  where $\pi_{G_1}$ denotes an irreducible, non-admissible representation of $G_1$ with trivial action of the finite group $Z_{G_1}$ constructed in Lemma \ref{lem1}.  The representation $\pi'$ is smooth, irreducible, and non-admissible.

  \vspace{5pt}

  \textit{Step 4.}  We extend the representation $\pi'$ to a representation of $T_{\alpha,1}G^\der$, and verify the extension has the desired properties.  We endow $\pi'$ with an action of $T_{\alpha,1}$ as follows: given $t \in T_{\alpha,1}$ and $f \boxtimes 1 \boxtimes \cdots \boxtimes 1 \in \pi'$ (with $f \in \pi_{G_1} = \Ind_{P_\alpha^- \cap G_1}^{G_1}(\rho'\boxtimes \chi')$, as in Lemma \ref{lem1}), we define
  $$t\cdot (f \boxtimes 1 \boxtimes \cdots \boxtimes 1) := f^t \boxtimes 1 \boxtimes \cdots \boxtimes 1$$
  where $f^t(g_1) = f(t^{-1}g_1t)$ for $g_1 \in G_1$.  When $t \in T_{\alpha,1} \cap G^\der$, we have $f^t(g_1) = f(g_1t_1)$, where $t_1$ denotes the projection of $t$ to the $G_1$-component; this relation follows from the fact that $t_1$ lies in the pro-$p$ part of the center of $M_\alpha \cap G_1$, and therefore acts trivially on $\rho'\boxtimes\chi'$.  One then checks that the action of $T_{\alpha,1}$ correctly intertwines with the given action of $G^\der$, and we therefore obtain an action of $T_{\alpha,1}G^\der$ on $\pi'$.

  To see that the above action is smooth, fix $f \in \pi_{G_1}$.  By the Iwasawa decomposition, we may write any element of $G_1$ as $g_1 = m_1 n_1^- k_1$, where $m_1 \in M_\alpha \cap G_1$, $n_1^- \in N_\alpha^- \cap G_1$ and $k_1 \in \bG_1(\cO_F)$.  For $t \in T_{\alpha,1}$, we have
  $$f^t(g_1)  =  f(t^{-1}g_1t)  =  f(t^{-1}m_1 n_1^-tt^{-1}k_1t)  =  \left((\rho'\boxtimes\chi')(t^{-1}m_1 t)\right)\cdot (f(k_1k_1^{-1}t^{-1}k_1t)).$$
  Since $t \in T_{\alpha,1}$, it centralizes $M_\alpha \cap G_1$, and therefore $t^{-1}m_1t = m_1$.  Furthermore, if $t$ lies in the $\ell^{\textnormal{th}}$ congruence subgroup of $T_{\alpha}$, \cite[Prop. 13.2.5(4)]{kalethaprasad} implies that $k_1^{-1}t^{-1}k_1t$ lies in the $\ell^{\textnormal{th}}$ congruence subgroup of $\bG_1(\cO_F)$.  In particular, since the stabilizer of $f$ is open, we have $k_1^{-1}t^{-1}k_1t \in \textnormal{Stab}_{G_1}(f)$ for $t$ lying in a sufficiently small open subgroup of $T_{\alpha,1}$.  In this case, the above equalities become
  $$f^t(g_1) = \left((\rho'\boxtimes\chi')(m_1)\right)\cdot (f(k_1)) = f(m_1n_1^-k_1) = f(g_1),$$
  which shows that the $T_{\alpha,1}$-action is smooth.

  The $T_{\alpha,1}G^\der$ representation $\pi'$ is also irreducible, since the restriction to $G^\der$ is.  Finally, we verify that $\pi'$ is non-admissible as a representation of $T_{\alpha,1}G^\der$.  Suppose $f \boxtimes 1 \boxtimes \cdots \boxtimes 1 \in \pi'^{\bG^\der(\cO_F)_1}$, where $\bG^\der(\cO_F)_1$ denotes the first congruence subgroup of $\bG^\der(\cO_F)$.  By the argument above, the element $k_1^{-1}t^{-1}k_1t$ lies in $\bG^\der(\cO_F)_1$ for any $t \in T_{\alpha,1}$ and any $k_1 \in \bG_1(\cO_F)$.  Thus, we obtain $t\cdot (f \boxtimes 1 \boxtimes \cdots \boxtimes 1) = f \boxtimes 1 \boxtimes \cdots \boxtimes 1$, which shows $\pi'^{\bG^\der(\cO_F)_1} \subset \pi'^{T_{\alpha,1}\bG^\der(\cO_F)_1}$.  We therefore conclude that $\pi'$ is non-admissible as a $T_{\alpha,1}G^\der$-representation.

  \vspace{5pt}

  \textit{Step 5.}  We extend the representation $\pi'$ to a representation of $T_{\alpha,1}Z^\circ G^\der$, and verify the extension has the desired properties.  We do this by decreeing that $Z^\circ$ acts trivially on $\pi'$.  To see that this action is well-defined, we must verify that any element of $T_{\alpha,1}G^\der$ which also lies in $Z^\circ$ acts trivially (that is, that the actions agree on $Z^\circ \cap T_{\alpha,1}G^\der$).  Let $z \in Z^\circ \cap T_{\alpha,1}G^\der$, and write $z = z'z_1$, where $z_1 \in Z^\circ_1$ and $z'$ lies in (the image of) $X_*(\bZ^\circ) \otimes_{\bbZ}(\varpi^{\bbZ} \times k_F^\times) \longhookrightarrow X_*(\bZ^\circ) \otimes_{\bbZ} F^\times \stackrel{\sim}{\longrightarrow} Z^\circ$.  We claim that the image of $z'$ in $S$ is trivial.  Indeed, since $z_1 \in Z^\circ_1 \subset T_{\alpha,1}$, we have $z' = z_1^{-1}z \in T_{\alpha,1}G^\der$, and therefore the image of $z'$ in $S$ lies in $S_1$ (recall that $T_{\alpha,1}$ surjects onto $S_1$).  However, since $z' \in X_*(\bZ^\circ) \otimes_{\bbZ}(\varpi^{\bbZ} \times k_F^\times)$, its image in $S$ will lie in $X_*(\bS) \otimes_{\bbZ}(\varpi^{\bbZ} \times k_F^\times)$, which has trivial intersection with $S_1$.  We conclude that $z' \in Z^\circ \cap G^\der$, and therefore $Z^\circ \cap T_{\alpha,1}G^\der = (Z^\circ \cap G^\der)Z_1^\circ$.  Now, the group $Z^\circ \cap G^\der$ is finite and contained in the center $Z_{G^\der} \cong Z_{G_1} \times \cdots \times Z_{G_n}$.  Therefore, by our choice of $\pi_{G_1}$ we see that $Z^\circ \cap G^\der$ acts trivially on $\pi'$.  Likewise, $Z^\circ_1$ acts trivially on $\pi'$ by definition of the $T_{\alpha,1}$-action.  We conclude that $\pi'$ extends to a representation of $T_{\alpha,1}Z^\circ G^\der$, and it is readily checked that it is smooth, irreducible, and non-admissible.

  \vspace{5pt}

  \textit{Step 6.}  We construct the desired representation $\pi$ of $G$.  As in Step 1, we note that $\Ind_{T_{\alpha,1}Z^\circ G^\der}^G(\pi')$ is a smooth, non-admissible $G$-representation of finite length (since $T_{\alpha,1}Z^\circ G^\der$ is of finite index in $G$).  Further, at least one Jordan--H\"older factor of $\Ind_{T_{\alpha,1}Z^\circ G^\der}^G(\pi')$ must be non-admissible.  Defining $\pi$ to be this subquotient gives a smooth, irreducible, non-admissible $G$-representation with trivial action of $Z^\circ$.  
\end{proof}

\begin{rmk}
  When $\textnormal{char}(F) = 0$, the above proof can be substantially shortened: in this case, the group scheme $\bZ^\circ \cap \bG^\der$ is always smooth over $F$, and the argument of Step 1 above implies the group $Z^\circ G^\der$ is always of finite index in $G$.  Thus, there is no need to use the subgroup $T_{\alpha,1}$ (as in the subsequent steps).  On the other hand, when $\textnormal{char}(F) = p$ and $G = \nG\nL_p(F)$, we have $G/Z^\circ G^\der = \nG\nL_p(F)/F^\times \nS\nL_p(F) \cong F^\times/(F^\times)^p$, which is of infinite order, and the simpler strategy no longer works. 
\end{rmk}

\begin{thm}
  \label{mainthm}
  Suppose $\bG$ is an arbitrary split connected reductive group over $F$, and satisfies $\bG \neq \bT$.  Then $G$ possesses a smooth, irreducible, non-admissible representation $\pi$ defined over $\overline{k_F}$.
\end{thm}

\begin{proof}
  Let 
  $$1 \longrightarrow \widetilde{\bZ} \longrightarrow \widetilde{\bG} \longrightarrow \bG \longrightarrow 1$$
  denote a $z$-extension of $\bG$.  Thus, $\widetilde{\bG}$ is a split connected reductive group with simply connected derived subgroup, and $\widetilde{\bZ}$ is a split torus contained in the connected center of $\widetilde{\bG}$.  In particular, by Hilbert's Theorem 90, taking $F$-rational points of the short exact sequence above is exact, and we obtain a short exact sequence of topological groups
  $$1 \longrightarrow \widetilde{Z} \longrightarrow \widetilde{G} \longrightarrow G \longrightarrow 1.$$
  By Lemma \ref{lem2}, the group $\widetilde{G}$ admits a smooth, irreducible, non-admissible representation $\pi$ with a trivial action of the $F$-points of the connected center of $\widetilde{\bG}$.  In particular, the group $\widetilde{Z}$ acts trivially on $\pi$, and therefore $\pi$ descends to a representation of $\widetilde{G}/\widetilde{Z} \cong G$ having the same properties.
\end{proof}

\begin{cor}
  \label{maincor}
  Let $k$ be a field of characteristic $p$, and suppose $\bG$ is an arbitrary split connected reductive group over $F$ which satisfies $\bG \neq \bT$.  Then $G$ possesses a smooth, irreducible, non-admissible representation $\pi$ defined over $k$.
\end{cor}

See also \cite[Rmk. 3]{ghatesheth}.  

\begin{proof}
  We first claim that the representation $\pi$ constructed in Theorem \ref{mainthm} has an absolutely irreducible model $\pi_0$ over some finite extension $k_F''$ of $k_F$.  Indeed:
  \begin{itemize}[label=$\diamond$]
    \item In Step 1 of Lemma \ref{lem1}, the representation $\rho|_{Z_{\nG\nL_2}\nS\nL_2(F)}$ has a model $\rho_0|_{Z_{\nG\nL_2}\nS\nL_2(F)}$ over $k_F$, which is a direct sum of at most 4 irreducible $Z_{\nG\nL_2}\nS\nL_2(F)$-representations over $k_F$. Therefore, the endomorphism ring $\textnormal{End}_{k_F[Z_{\nG\nL_2}\nS\nL_2(F)]}(\rho_0|_{Z_{\nG\nL_2}\nS\nL_2(F)})$ is a finite-dimensional, semisimple $k_F$-algebra, and hence by the Artin--Wedderburn theorem is isomorphic to a finite product of matrix rings over finite extensions of $k_F$.  Let $k_F'$ denote the compositum of all these finite extensions; then $k_F'$ is finite over $k_F$ and  $(\rho_0 \otimes_{k_F} k_F')|_{Z_{\nG\nL_2}\nS\nL_2(F)}$ is a $k_F'$-model of $\rho|_{Z_{\nG\nL_2}\nS\nL_2(F)}$, all of whose irreducible summands are absolutely irreducible (its endomorphism ring is a finite product of matrix rings over $k_F'$).  It therefore contains a $k_F'$-model $\pi_0$ of the representation $\pi$ constructed in Step 1 of Lemma \ref{lem1}.  
    \item In Step 3 of Lemma \ref{lem1}, define $\pi_0 := \Ind_{P_\alpha^-}^G((\rho_0\otimes_{k_F}\omega_0^{-m}\circ\det_{\nG\nL_2})\boxtimes_{k_F}\chi_0')$, where $\omega_0$ and $\chi_0'$ denote the $k_F$-valued versions of the characters $\omega$ and $\chi'$.  By \cite[Prop. III.12(i)]{henniartvigneras:rationality}, $\pi_0$ is a $k_F$-model of $\pi$, and is therefore irreducible.  Moreover, using the right adjoint of parabolic induction (see \cite[Thm. 5.3(2)]{vigneras:rightadjoint}), we have 
    $$\End_{k_F[G]}(\pi_0) \cong \End_{k_F[M_\alpha]}((\rho_0 \otimes_{k_F} \omega_0^{-m}\circ\sideset{}{_{\nG\nL_2}}{\det})\boxtimes_{k_F} \chi_0') \cong \End_{k_F[\nG\nL_2(F)]}(\rho_0) = k_F,$$ 
    so that $\pi_0$ is indeed absolutely irreducible.
    \item By the previous points, the representation $\pi'$ constructed in Steps 1 and 3 through 6 of Lemma \ref{lem2} has an absolutely irreducible model $\pi'_0$ over some finite extension $k_F'$ of $k_F$.  Let $H$ denote either $Z^\circ G^\der$ (in the setting of Step 1) or $T_{\alpha,1}Z^\circ G^\der$ (in the setting of Steps 3 to 6), so that $H$ is a finite-index, normal subgroup of $G$.  Consider the representation $\Ind_H^G(\pi_0')$ over $k_F'$, and let $\tau_0$ denote an irreducible subquotient.  Then $\tau_0|_{H}$ is a subquotient of $\Ind_H^G(\pi_0')|_H$, which by the Mackey formula is isomorphic to $\bigoplus_{g \in G/H}\pi_0'^g$, a direct sum of absolutely irreducible representations over $k_F'$.  Thus, $\tau_0|_H$ is semisimple, and $\End_{k_F'[H]}(\tau_0|_H)$ is a subring of $\End_{k_F'[H]}(\Ind_H^G(\pi_0')|_H)$, which is finite-dimensional over $k_F'$.  We conclude that $\End_{k_F'[G]}(\tau_0)$, being a subring of $\End_{k_F'[H]}(\tau_0|_{H})$, is a finite-dimensional division algebra over $k_F'$, that is, a finite extension of $k_F'$.  Let $k_F''$ denote the compositum of all such finite extensions as $\tau_0$ ranges over the (finitely many) irreducible subquotients of $\Ind_H^G(\pi_0')$.  Then $k_F''$ is finite over $k_F'$ and $\tau_0\otimes_{k_F'}k_F''$ is a direct sum of absolutely irreducible representations (it is semisimple by \cite[\S 12, No. 5, Cor.]{bourbaki:algebra-ch8}, and its endomorphism ring is $\End_{k_F'[G]}(\tau_0)\otimes_{k_F'}k_F''$, a finite product of copies of the field $k_F''$).  Since this holds for every irreducible subquotient of $\Ind_H^G(\pi_0')$, we conclude that every irreducible subquotient of $\Ind_H^G(\pi_0'\otimes_{k_F'}k_F'')$ is absolutely irreducible.  In particular, $\Ind_H^G(\pi_0'\otimes_{k_F'}k_F'')$ contains as a subquotient a $k_F''$-model $\pi_0$ of the representation $\pi$ constructed in Step 1 or Step 6 of Lemma \ref{lem2}.   
    \item In Theorem \ref{mainthm}, any absolutely irreducible $k_F''$-model $\pi_0$ of the $\widetilde{G}$-representation $\pi$ descends to the group $G$.
  \end{itemize}

  Recall that $G$ is the group of $F$-rational points of an arbitrary split connected reductive group, and $k$ is an arbitrary field of characteristic $p$.  Fix an absolutely irreducible $k_F''$-model $\pi_0$ of the $G$-representation $\pi$ of Theorem \ref{mainthm}.  Let $\overline{k}$ denote an algebraic closure of $k$ and fix compatible inclusions among all coefficient fields appearing.  We have $$\pi\otimes_{\overline{k_F}}\overline{k} \cong (\pi_0 \otimes_{k_F''} \overline{k_F}) \otimes_{\overline{k_F}}\overline{k} \cong (\pi_0 \otimes_{k_F''} k_F''k) \otimes_{k_F'' k}\overline{k},$$ 
  and therefore $\pi_0 \otimes_{k_F''} k_F''k$ defines a descent to $k_F''k$ of the absolutely irreducible $G$-representation $\pi\otimes_{\overline{k_F}}\overline{k}$.  Since $k_F''$ is a finite extension of $\bbF_p$, $k_F''k$ is a finite extension of $k$, and \cite[Lem. II.5]{henniartvigneras:rationality} implies that there exists a smooth, irreducible $G$-representation $\Pi$ over $k$ such that $\pi\otimes_{\overline{k_F}}\overline{k}$ is isomorphic to a subrepresentation of $\Pi\otimes_{k}\overline{k}$.  Taking invariants by an open pro-$p$ subgroup (and applying \cite[Lem. III.1(ii)]{henniartvigneras:rationality}) shows that $\Pi$ is non-admissible.  
\end{proof}

\section{A non-equivalence result}

We now deduce the main consequence of the above results to the study of pro-$p$-Iwahori--Hecke algebras.  Recall that $F$ is a nonarchimedean local field which is either in case \eqref{caseA} or \eqref{caseB}, that $k_F$ denotes the residue field of $F$, and $k$ is a coefficient field of characteristic $p$.  

\begin{thm}
  \label{thm:hecke-alg-cor}
  Suppose $\bG$ is an arbitrary split connected reductive group over $F$ with $\bG \neq \bT$, and let $I_1$ denote a choice of pro-$p$-Iwahori subgroup of $G$.  Then the functor of $I_1$-invariants $\pi \longmapsto \pi^{I_1}$ does \textbf{not} induce an equivalence of categories between $\mathfrak{Rep}^{I_1}_k(G)$ and $\mathfrak{Mod-}\cH_k(G,I_1)$.  
\end{thm}

\begin{proof}
  Suppose on the contrary that $\pi \longmapsto \pi^{I_1}$ yields the desired equivalence.  In particular, the functor of $I_1$-invariants maps simple objects in $\mathfrak{Rep}^{I_1}_k(G)$ to simple objects in $\mathfrak{Mod-}\cH_k(G,I_1)$.  Corollary \ref{maincor} implies that $\mathfrak{Rep}^{I_1}_k(G)$ contains an irreducible $G$-representation $\pi$ over $k$ which is non-admissible.  Thus, $\pi^{I_1}$ gives an infinite-dimensional simple $\cH_k(G,I_1)$-module.  However, the proof of \cite[Lem. 6.9]{ollivierschneider} implies that every simple $\cH_k(G,I_1)$-module is finite-dimensional, and we therefore arrive at a contradiction.  
\end{proof}

\begin{cor}
  Suppose $p \geq 7$, and $F$ is a nonarchimedean local field of residual characteristic $p$.  Then the functor of $I_1$-invariants $\pi \longmapsto \pi^{I_1}$ induces an equivalence of categories between $\mathfrak{Rep}^{I_1}_k(\textnormal{GL}_2(F))$ and $\mathfrak{Mod-}\cH_k(\textnormal{GL}_2(F),I_1)$ if and only if $F = \bbQ_p$.
\end{cor}

\begin{proof}
  When $F \neq \bbQ_p$, this follows from the above theorem and \cite[Thm. 1.1, Rmk. 1.2]{le-singh}.  When $F = \bbQ_p$, this follows from \cite[\S 1, Thm.]{ollivier:foncteur} (using \cite[Thm. 5.2]{koziol:glnsln} to remove the condition on central characters).
\end{proof}

\bibliographystyle{amsalpha}
\bibliography{refs}

\end{document}